\documentclass[12pt]{article}
\usepackage{amsmath,amssymb,amsthm,amscd,latexsym,}
 
\usepackage{graphicx}

\makeatletter
\renewcommand{\mod}[1]{\allowbreak \if@display \mkern 8mu \else
\mkern 5mu\fi {\operator@font mod}\,\,#1}
\makeatother

\newcommand{\bc}{\mathbb C}

 \newcommand{\bq}{\mathbb Q}
\newcommand{\br}{\mathbb R}
\newcommand{\bz}{\mathbb Z}

\newcommand{\bp}{\mathbb P}

\DeclareMathOperator{\Aut}{{\rm Aut}\,}

\DeclareMathOperator{\rk}{rk}

\DeclareMathOperator{\cha}{{\rm char}\,}
\newtheorem{theorem}{Theorem}

\newtheorem{definition}[theorem]{Definition}
\newtheorem{corollary}[theorem]{Corollary}

\newcommand\La{\mathcal L}

\newcommand\M{\mathcal M}

\begin{document}
\title{Elliptic fibrations on K3 surfaces}
\date{}
\author{Viacheslav V. Nikulin}
\maketitle

\begin{abstract} The present paper consists mainly of a review and 
applications of our old results related to the title. 
We discuss how many elliptic fibrations
and elliptic fibrations with infinite automorphism groups
(or Mordell--Weil groups) an algebraic K3 surface over an algebraically
closed field can have.

As examples of applications of the same ideas, we also consider 
K3 surfaces with exotic structures: with finite number of non-singular
rational curves, with finite number of Enriques involutions,
and with naturally arithmetic automorphism groups. 
\end{abstract}

\centerline{Dedicated to my old friend and colleague Slava Shokurov}

\centerline{on occasion of his 60th birthday} 

\section{Introduction}
\label{introduction}

This is mainly a review and
applications of our old results related to elliptic fibrations on K3
surfaces over algebraically closed fields. See \cite{Nik1}---\cite{Nik8}.
The most important are our papers \cite{Nik2}, \cite{Nik3}, \cite{Nik7} and
\cite{Nik8}.

This was the subject of our talk at the Oberwolfach workshop 
``Higher dimensional elliptic fibrations'' in October 2010.
Elliptic fibrations are especially interesting for Fano and Calabi--Yau
varieties. Thus, it is interesting to study these fibrations 
in the case of K3 surfaces which are 2-dimensional Calabi--Yau manifolds.

We consider algebraic K3 surfaces $X$ over arbitrary algebraically
closed fields $k$. 

In Section \ref{sec1}, we discuss basic results by Piatetsky-Shapiro
and Shafarevich \cite{PS}. In particular, we discuss,   
when a K3 surface $X$ has an elliptic fibration.

In Section \ref{sec2}, we discuss, when a K3 surface $X$ has an elliptic
fibration with infinite automorphism group (or the Mordell--Weil group).
See \cite{Nik2}.

In Section \ref{sec3}, we discuss our general results from
\cite{Nik2}, \cite{Nik7} and \cite{Nik8} on existence of non-zero  
exceptional elements of the Picard lattice with respect to 
the automorphism group of a K3 surface. Here an element
$x$ of the Picard lattice $S_X$ is called {\it exceptional} with respect
to the automorphism group $\Aut X$, if its orbit $\Aut X(x)$ in $S_X$
is finite. These results will give the main tools for further applications.

In Section \ref{sec4} (see also Section \ref{sec3}),  
we discuss, how many elliptic fibrations and elliptic 
fibrations with infinite automorphism group
a K3 surface can have. In particular, for the Picard number
$\rho(X)\ge 3$, we show that a K3 surface $X$ has infinite number of
elliptic fibrations and infinite number of elliptic fibrations with
infinite automorphism groups if it has one of them,
except a finite number of exceptional Picard lattices $S_X$. This is
mainly related to our results in \cite{Nik2}, \cite{Nik3}, 
\cite{Nik7} and \cite{Nik8}.

As examples of applications of the same ideas, in Section \ref{sec5}, 
we consider K3 surfaces with exotic structures: finite number of non-singular
rational curves, finite number of Enriques involutions, and with
naturally arithmetic automorphism groups.

We thank the referee for the careful reading of the paper and 
important remarks. 

\section{Results by Piatetsky-Shapiro  
and\\ 
Shafarevich about existence of\\ 
elliptic fibrations on K3 surfaces}\label{sec1}

We remind that an algebraic K3 surfaces $X$ is a non-singular projective
algebraic surface over an algebraically closed field $k$ such that the 
canonical class $K_X=0$ and the irregularity 
$q(X)=\dim H^1(X, {\mathcal O}_X)=0$. 

Further in this section,
$X$ is an algebraic K3 surface over an algebraically closed field.
We denote by $S_X$ the Picard lattice of $X$. It is well-known that $S_X$ is
a hyperbolic (i.e., of signature $(1,\rho(X)-1)$) even integral
lattice of rank $\rho(X)$ where $\rho(X)=\rk S_X$ is the Picard
number of $X$. It can be a very arbitrary even hyperbolic 
lattice of rank $\rho(X)\le 22$, and it is an 
important invariant of $X$.
It will be the most important for us.

According to Piatetsky-Shapiro and
Shafarevich \cite{PS}, elliptic fibrations on $X$ are in one to one
correspondence with primitive isotropic numerically
effective (i.e., $nef$) elements $c\in S_X$. That is $c\not=0$, $c^2=0$;
$c/n\in S_X$ only for integers
$n=\pm 1$; $c\cdot D\ge 0$ for any effective divisor $D$ on $X$. For such
$c\in S_X$, the complete
linear system $|c|$ is one-dimensional without base points, and it gives
an elliptic fibration $|c|:X\to \bp^1$,  that is the general
fibre is an elliptic curve (for $\cha k=2$ or $3$ it can be quasi-elliptic,
see \cite{RS}).

The following facts were also observed in \cite{PS}.
By the Riemann-Roch Theorem
for surfaces, any irreducible curve $D$ on $X$ with negative square
has $D^2=-2$, and it is then
rational non-singular, hence $\bp^1$. It follows that the $nef$ cone
$NEF(X)\subset V^+(X)\subset S_X\otimes \br$
(or $\M(X)=NEF(X)/\br^+\subset \La(S_X)=V^+(X)/\br^+$) is a fundamental
chamber for the
reflection group $W^{(2)}(S_X)\subset O(S_X)$ generated by 2-reflections
$s_\delta:x\to x+(x\cdot \delta)\delta$ in
elements $\delta\in S_X$ with $\delta^2=-2$. Moreover, classes of
non-singular rational curves on $X$ are in one
to one correspondence to elements $\delta\in S_X$
with $\delta^2=-2$ which are perpendicular to
codimension one faces of $\M(X)$ and directed outwards. See \cite[Sec. 3]{RS}.
We denote this set by $P(\M(X))$.
Here $\br^+$ is the set of all positive real numbers,
$V^+(X)$ is a half-cone of the cone $V(S_X)$ of elements of
$S_X\otimes \br$
with positive square, and $\La(S_X)$ is the hyperbolic
space related to $S_X$ or $X$. We denote by
$$
A(\M(X))=\{\phi \in O(S_X)\ |\ \phi(V^+(X))=V^+(X)\ \text{and}\
\phi(\M(X))=\M(X)\}
$$
the symmetry group of $\M(X)$, and then $\{\pm 1\}W^{(2)}(S_X)\rtimes
A(\M(X))=O(S_X)$ is the semi-direct product.
By the theory of arithmetic groups (or integral quadratic forms theory),
then the fundamental
domain for $O(S_X)$ is the same as the fundamental domain for $A(\M(X))$ in
$\M(X)$. In  particular, this fundamental domain is 
a finite rational polyhedron.

It follows that {\it there exists only a finite number of
elliptic pencils on $X$
up to the action of $A(\M(X))$}. Similarly, {\it there exists only a 
finite number of non-singular rational curves on $X$
up to the action of $A(\M(X))$.} Moreover, for any isotropic element
$c^\prime\in S_X$, there exists $w\in W^{(2)}(S_X)$
such that $\pm w(c^\prime)$ is $nef$. Thus {\it $X$ has an elliptic fibration
if and only if the Picard lattice
$S_X$ represents zero: there exists $0\not=x\in S_X$ with $x^2=0$}.
In particular, this is valid if $\rho(X)\ge 5$.

The fundamental result of \cite{PS} which follows from the Global
Torelli Theorem for K3 surfaces (also proved in \cite{PS}) is
that the action of $\Aut X$ in $S_X$ has only finite kernel (see also
\cite{RS} if $\cha k>0$), and for $\cha k=0$
it gives a finite index subgroup in $A(\M(X))$. In particular, for
$\cha k=0$, up to finite groups,
we have natural isomorphisms of groups:
$$
\Aut X\thickapprox A(\M(X))\cong O^+(S_X)/W^{(2)}(S_X)
$$
where $O^+(S_X)=\{\phi \in O(S_X)\ |\ \phi(V^+(X))=V^+(X)\}$
is the subgroup of $O(S_X)$ of index $2$.
It follows that {\it for $\cha k=0$, a K3 surface $X$ has
only finite number of elliptic fibrations and non-singular rational curves
up to the action of the automorphism
group $\Aut X$. } This is the same as for all $nef$ elements $h\in S_X$ with a
fixed positive square $h^2>0$.

\section{Existence of elliptic fibrations with infinite automorphism groups
on K3 surfaces}
\label{sec2} Further, $X$ is a
K3 surface over an algebraically closed field $k$.

Let $c\in S_X$ be a primitive isotropic $nef$ element.
By the theory of elliptic surfaces, see e.g.\cite[Ch. VII]{S}
(or by Global Torelli Theorem for K3 surfaces, if $\cha k=0$),
the group $\Aut(c)$ of automorphisms of of the elliptic fibration
$|c|:X \to \bp^1$ is, up to finite index, the abelian
group $\bz^{r(c)}$ where 
\begin{equation}
r(c)=\rk c^\perp-\rk (c^\perp)^{(2)}.
\label{r}
\end{equation}
Here $c^\perp$ is the orthogonal complement to $c$ in $S_X$
(obviously, $\rk c^\perp=\rho(X)-1$), and
the sublattice $(c^\perp)^{(2)}\subset c^\perp$ is generated by $c$ and
by all elements with square $(-2)$ in $c^\perp$.
Equivalently, $(c^\perp)^{(2)}$ is generated by all irreducible components
of fibres of $|c|:X\to \bp^1$. In particular, $\Aut(c)$
is finite if and only if either $\rho(X)=2$, or $c^\perp$ is generated by
$c$ and by  elements with square $(-2)$, up to finite index.
Up to finite index, $\Aut(c)$ is the same as the Mordell-Weil group of
the elliptic fibration $c$ when we consider only automorphisms
from $\Aut(c)$ which act trivially on the base $\bp^1$. 

\medskip

The interesting question is:

\medskip 

{\it When does $X$ have elliptic fibrations with infinite automorphism groups?}

\medskip 

It is important, for example, for studying the dynamics of $\Aut X$ 
(e.g. see \cite{Can}) and the arithmetic of $X$.

The main obstruction for the existence of the fibrations in question 
is the finiteness of the automorphism group $\Aut X$ of $X$.
Indeed, if $\Aut X$ is finite, then automorphism groups of all
elliptic fibrations $c$ on $X$ are also finite since $\Aut(c)\subset \Aut X$.

Surprisingly, for $\rho(X)\ge 6$, this obstruction is sufficient
and necessary according to \cite{Nik2}, and this is valid
for $k$ of any characteristic. 
These results can be formulated for arbitrary hyperbolic lattices $S$ if
one fixes a fundamental chamber $\M\subset \La(S)$
for $W^{(2)}(S)$ and considers fundamental primitive isotropic
elements $c\in S$ that is $\br^+c\in \overline{\M}$.
Instead of $\Aut X$ one should consider the symmetry group
$A(\M)\subset O^+(S)$ or $O^+(S)/W^{(2)}(S)$.

By \eqref{r}, all elliptic fibrations on $X$ have finite automorphism groups
if and only if the hyperbolic
lattice $S=S_X$ satisfies the property:
\begin{equation}
\rk (c^\perp)=\rk (c^\perp)^{(2)}\ \ \ \text{for any isotropic\ \ } c\in S.
\label{finel}
\end{equation}

We have the following results from \cite{Nik2}.

\begin{theorem} Let $S$ be an even hyperbolic lattice of rank
$\rho=\rk S\ge 6$ (respectively, $X$ is
a K3 surface over an algebraically closed field, and $\rho(X)\ge 6$).
Then the following conditions
 (a), (b), (c) below are equivalent:

 (a) $S$ satisfies condition \eqref{finel} (respectively, automorphism
 groups of all elliptic fibrations on $X$
 are finite).

(b) The group $A(\M)\cong O^+(S)/W^{(2)}(S)$ is finite,
(respectively, $\Aut X$ is finite).

(c) The lattice $S$ belongs to the finite list of even hyperbolic
lattices below found in \cite{Nik2} (respectively,
$S=S_X$ is one of the lattices from the list). 
\label{elK36}
\end{theorem}

The list of lattices found in \cite{Nik2} is the following
(we use notations from \cite{Nik1} and \cite{Nik2}, which
are now standard; $\oplus$ is orthogonal sum of lattices, $U$ is the 
even unimodular lattice of signature $(1,1)$, $A_n$, $D_m$ and $E_k$ 
are negative definite root lattices corresponding to root systems 
${\mathbb A}_n$, ${\mathbb D}_m$ and ${\mathbb E}_k$ respectively,  
$S(\lambda)$ is obtained from 
a lattice $S$ by multiplication of its form by $\lambda \in \bz$, 
$\langle A \rangle$ is a lattice with the matrix $A$ in some basis): 

\medskip

\noindent
{\it The list of all even hyperbolic lattices $S$ with
$[O(S):W^{(2)}(S)]<\infty$ and $\rk S\ge 6$ (see \cite{Nik2}):}

\medskip

\noindent
$S=U\oplus 2E_8\oplus A_1$; 
$U\oplus 2E_8$; 
$U\oplus E_8\oplus E_7$;
$U\oplus E_8\oplus D_6$; 
$U\oplus E_8\oplus D_4\oplus A_1$;
$U\oplus E_8\oplus D_4$, $U\oplus D_8\oplus D_4$, $U\oplus E_8\oplus 4A_1$;
$U\oplus E_8\oplus 3A_1$, $U\oplus D_8\oplus 3A_1$, $U\oplus A_3\oplus E_8$;
$U\oplus E_8\oplus 2A_1$, $U\oplus D_8\oplus 2A_1$, 
$U\oplus D_4\oplus D_4\oplus 2A_1$, $U\oplus A_2\oplus E_8$;
$U\oplus E_8\oplus A_1$, $U\oplus D_8\oplus A_1$, 
$U\oplus D_4\oplus D_4\oplus A_1$, $U\oplus D_4\oplus 5A_1$; 
$U\oplus E_8$, $U\oplus D_8$, $U\oplus E_7\oplus A_1$,
$U\oplus D_4\oplus D_4$, $U\oplus D_6\oplus 2A_1$,
$U(2)\oplus D_4\oplus D_4$, $U\oplus D_4\oplus 4A_1$,
$U\oplus 8A_1$, $U\oplus A_2\oplus E_6$; 
$U\oplus E_7$, $U\oplus D_6\oplus A_1$, 
$U\oplus D_4\oplus 3A_1$, $U\oplus 7A_1$, $U(2)\oplus 7A_1$,
$U\oplus A_7$, $U\oplus A_3\oplus D_4$, $U\oplus A_2\oplus D_5$,
$U\oplus D_7$, $U\oplus A_1\oplus E_6$; 
$U\oplus D_6$, $U\oplus D_4\oplus 2A_1$,
$U\oplus 6A_1$, $U(2)\oplus 6A_1$,
$U\oplus 3A_2$, $U\oplus 2A_3$, $U\oplus A_2\oplus A_4$,
$U\oplus A_1\oplus A_5$, $U\oplus A_6$, $U\oplus A_2\oplus D_4$,
$U\oplus A_1\oplus D_5$, $U\oplus E_6$;
$U\oplus D_4\oplus A_1$, $U\oplus 5A_1$,
$U(2)\oplus 5A_1$, $U\oplus A_1\oplus 2A_2$, $U\oplus 2A_1\oplus A_3$,
$U\oplus A_2\oplus A_3$, $U\oplus A_1\oplus A_4$, $U\oplus A_5$, $U\oplus D_5$;
$U\oplus D_4$, $U(2)\oplus D_4$, $U\oplus 4A_1$,
$U(2)\oplus 4A_1$,
$U\oplus 2A_1\oplus A_2$, $U\oplus 2A_2$, $U\oplus A_1\oplus A_3$,
$U\oplus A_4$, $U(4)\oplus D_4$, $U(3)\oplus 2A_2$.

\vskip0.5cm

Thus, a K3 surface $X$ over an algebraically closed field and with
$\rho(X)\ge 6$ has an elliptic fibration with infinite automorphism
group if and only if its Picard lattice $S_X$ is different from each
lattice of this finite list.
If the Picard lattice $S_X$ of $X$ is one of lattices from the list,
then not only automorphism groups of
all elliptic fibrations on $X$ are finite, but the full automorphism
group $\Aut X$ is finite.

\vskip0.5cm

If $\rk S=5$, then similar theorem is valid if one excludes two infinite
series of even hyperbolic lattices, see \cite{Nik2}.

\begin{theorem} Let $S$ be an even hyperbolic lattice of the rank
$\rk S=5$ and $S$ is different from lattices
$\langle 2^m\rangle\oplus D_4$, $m\ge 5$, and
$\langle 2\cdot 3^{2n-1}\rangle \oplus 2A_2$, $n\ge 2$
(respectively, a K3 surface $X$ over an algebraically closed field has
$\rho(X)=5$ and $S_X$ is different from the lattices of these two series).

Then the following conditions (a), (b), (c) below are equivalent: 

 (a) $S$ satisfies the condition \eqref{finel} (respectively, automorphism
 groups of all elliptic fibrations on $X$
 are finite).

(b) The group $A(\M)\cong O^+(S)/W^{(2)}(S)$ is finite, (respectively,
$\Aut X$ is finite).

(c) The lattice $S$ belongs to the finite list of even hyperbolic lattices
of rank $5$ below found in \cite{Nik2} (respectively,
$S_X$ is one of the lattices from this list).

\medskip

If $S$ is one of  lattices
$\langle 2^m\rangle\oplus D_4$, $m\ge 5$,
and $\langle 2\cdot 3^{2n-1}\rangle \oplus 2A_2$,
$n\ge 2$, then
$S$ satisfies \eqref{finel}, but the group
$A(\M)\cong O^+(S)/W^{(2)}(S)$ is infinite
(equivalently, if $S_X$ is one of lattices from these two series,
then all elliptic fibrations on $X$ have
finite automorphism groups, but $\Aut X$
is infinite if $\cha k=0$).
\label{elK35}
\end{theorem}

The list of lattices of rank 5 found in \cite{Nik2} is as follows:

\medskip

\noindent
{\it The list of all even hyperbolic lattices $S$
with $[O(S):W^{(2)}(S)]<\infty$
and $\rk S=5$ (see \cite{Nik2}):}

\medskip 

\noindent
$S=U\oplus 3A_1$, $U(2)\oplus 3A_1$, $U\oplus A_1\oplus A_2$,
$U\oplus A_3$, $U(4)\oplus 3A_1$,
$\langle 2^k\rangle \oplus D_4$, $k=2,\,3,\,4$,
$\langle 6 \rangle \oplus 2A_2$.

\medskip 

Thus, a K3 surface $X$ with $\rho(X)=5$ and any $\cha k$ has elliptic
fibrations with infinite automorphism
groups if and only if its Picard lattice $S_X$ is different from each
lattice of this finite list and
from the lattices of 2 infinite series in Theorem \ref{elK35}.
If the Picard lattice of $X$ is one of lattices from the finite list,
then not only automorphism groups of
all elliptic fibrations on $X$ are finite,
but the full automorphism group $\Aut X$ is finite.
If the Picard lattice of $X$ is one of lattices from the two infinite
series of lattices of Theorem \ref{elK35},
then the automorphism groups of all elliptic fibrations on $X$ are finite,
but $\Aut X$ is infinite
if $\cha k=0$ (if $\cha k>0$, it is not known).

\vskip0.5cm

If the Picard number $\rho (X)=4$ or $3$, no results, 
similar to that in Theorems 1 and 2, are known, except 
results which we shall cite below at the end of this section.  

\vskip0.5cm

If $\rho(X)=2$, then the automorphism groups of all elliptic
fibrations on $X$ are evidently finite.
If $\rho(X)=1$, then $X$ has no elliptic fibrations.

In particular, Theorems \ref{elK36} and \ref{elK35} describe all
even hyperbolic lattices $S$
having finite group $A(\M)\cong O^+(S)/W^{(2)}(S)$ (they are called
elliptically $2$-reflective) of rank $\rho=\rk S\ge 5$.
Similar finite description of elliptically $2$-reflective even hyperbolic
lattices was obtained for $\rho=4$ ($14$ lattices)
in \cite{Vin4} (see also \cite{Nik6}), and for $\rho=3$ ($26$ lattices)
in \cite{Nik5}.\footnote{We must correct the list of lattices in
\cite{Nik5}: the lattices $S_{6,1,2}^\prime$ and $S_{6,1,1}$ are
isomorphic.} Finiteness  was also generalized to arbitrary
arithmetic hyperbolic reflection groups and
corresponding reflective hyperbolic lattices over rings of integers
of totally real algebraic number fields.
See \cite{Nik3}, \cite{Nik4}, \cite{Nik6} and \cite{Vin2} 
(see also \cite{Vin3}).

\section{Elliptic fibrations with infinite\\
automorphism groups and  exceptional\\
elements in Picard lattices for K3 surfaces} \label{sec3}

Below, $X$ is a K3 surface over an algebraically closed field.

We consider the following general notion. For a hyperbolic lattice
$S$ and a subgroup $G\subset O(S)$, we
call $x\in S$ exceptional with respect to $G$ if its stabilizer
subgroup $G_x$ has finite index in $G$; equivalently,
the orbit $G(x)$ is finite. All exceptional elements with respect
to $G$ define exceptional sublattice $E\subset S$
with respect to $G$. Since $S$ is hyperbolic, logically the following
4 cases are possible:

(i) {\it Elliptic type of $G$.} The exceptional sublattice $E$ for $G$
is hyperbolic. Obviously, then $G$ is finite and
$E=S$. Then $E$ is called {\it hyperbolic.}

(ii) {\it Parabolic type of $G$.} The exceptional sublattice $E$ for
$G$ is semi-negative definite and has 1-dimensional kernel.
Then $E$ is called {\it parabolic.}

(iii) {\it Hyperbolic type of $G$.} The exceptional sublattice $E$ for $G$
is negative definite. Then $E$ is called {\it elliptic.}

(iv) {\it General hyperbolic type of $G$.} The exceptional sublattice $E$
for $G$ is zero.

Replacing $G$ by the action of $\Aut X$ in $S_X$,  we
obtain the following main definition.
An element of the Picard lattice $x\in S_X$ is
called {\it exceptional (with respect to $\Aut X$)} if its
stabilizer subgroup $(\Aut X)_x$ has finite index in $\Aut X$,
equivalently, the orbit $(\Aut X)(x)$ of $x$
is finite.

All exceptional elements of $S_X$ define a primitive sublattice
$E(S_X)$. We call it {\it the exceptional
sublattice of the Picard lattice (for $\Aut X$).}
This sublattice was introduced in \cite{Nik2} (it was denoted
as $R(S_X)$ in \cite{Nik2}), and the results
which we discuss below were mentioned and in fact proved in
\cite{Nik2} and \cite{Nik7}, \cite{Nik8} (see
\cite[Sect. 3]{Nik8}).
Below we just give more details.

Let us assume that $X$ has at least one elliptic fibration $c$
with infinite automorphism group. Then we have
the following statement where for a sublattice $F\subset S_X$ we
denote by $F_{pr}\subset S_X$ the primitive sublattice
$F_{pr}=S_X\cap (F\otimes \bq)\subset S_X\otimes \bq$ generated by $F$.

\begin{theorem} Let $X$ be a K3 surface over an algebraically closed field
which has at least one elliptic fibration with infinite automorphism
group.

Then the exceptional sublattice $E(S_X)$ is equal to
\begin{equation}
E(S_X)=
\bigcap_{c}{(c^\perp)^{(2)}}_{pr}
\label{exc}
\end{equation}
where $c$ runs through all elliptic fibrations on $X$ with infinite
automorphism groups (or the Mordell--Weil groups).

In particular, two exceptional sublattices of $S_X$, for $\Aut X$, and for the
subgroup of $\Aut X$ generated by Mordell--Weil groups of 
all elliptic fibrations with infinite 
automorphism groups on $X$, coincide.
\label{thexcX}
\end{theorem}

\begin{proof} Simple calculations, using theory of elliptic
surfaces (see \cite[Ch. VII]{S}), show (see \cite{Nik2}) that
exceptional elements
for $(\Aut X)_c$ (equivalently, for the Mordell--Weil group of the
elliptic fibration $|c|$) in $S_X$ define the sublattice
$(c^\perp)^{(2)}_{pr}$. It follows that $E(S_X)\subset Ell(S_X)$
where $Ell(S_X)$ is the right hand side of \eqref{exc}.

Since $X$ has at least one elliptic fibration with infinite automorphism
group and $S_X$ is hyperbolic, $Ell(S_X)$ is
either semi-negative definite with one dimensional kernel $\bz c$ 
(that is $Ell(S_X)$ is parabolic) when $X$ has only one elliptic 
fibration $c$ with infinite automorphism group, 
or $Ell(S_X)$ is negative definite
(that is $Ell(S_X)$ is elliptic) if $X$ has more than one elliptic fibrations 
with infinite automorphism
groups. In both cases, $\Aut X$ gives finite action on $Ell(S_X)$.
It follows that $Ell(S_X) \subset E(S_X)$. Thus, $E(S_X)=Ell(S_X)$.

This finishes the proof. 
\end{proof}

Like above, for an abstract hyperbolic lattice $S$ (replacing $S_X$),
a fundamental chamber $\M\subset \La(S)$
for the reflection group $W^{(2)}(S)$
(replacing $\M(X)=NEF(X)/\br^+\subset \La(S_X)$),
and for
the group $A(\M)$ of symmetries of $\M$ (replacing $\Aut X$),
we can similarly consider exceptional elements $x\in S$ for $A(\M)$
and the sublattice $E(S)\subset S$ of all exceptional elements for
$A(\M)$. For
a fundamental primitive isotropic element $c\in S$ for $\M$
(replacing an elliptic fibration of $X$),
we can similarly consider the stabilizer subgroup $A(\M)_c\subset A(\M)$
(replacing the automorphism group $\Aut(c)$ of the elliptic fibration
$c$ on $X$). Like for K3 surfaces, we have isomorphism up to finite groups
\begin{equation}
A(\M)_c\thickapprox \bz^{r(c)},\ \ \ r(c)=\rk c^\perp-\rk (c^\perp)^{(2)}.
\label{finfund}
\end{equation}

Using \eqref{finfund}, exactly the same considerations as for
Theorem \ref{thexcX} give similar result
for arbitrary hyperbolic lattices.

\begin{theorem} Let $S$ be a hyperbolic lattice over $\bz$,
$\M\subset \La(S)$ a fundamental chamber for $W^{(2)}(S)$
and $A(\M)\subset O^+(S)$ its symmetry group. Let us assume that $S$ has
at least one fundamental primitive isotropic
element $c$ with infinite stabilizer subgroup $A(\M)_c$.

Then the exceptional sublattice $E(S)$ for $A(\M)$ is equal to
\begin{equation}
E(S)=
\bigcap_{c}{(c^\perp)^{(2)}}_{pr}
\label{excS}
\end{equation}
where $c$ runs through all fundamental primitive isotropic elements
$c$ for $\M$ with
infinite stabilizer subgroups $A(\M)_c$.
\label{thexcS}
\end{theorem}

\medskip

For a K3 surface $X$ and $S_X$, we take $\M=\M(X)=NEF(X)/\br^+$.
By \cite{PS},
fundamental primitive isotropic elements for $\M$ and
elliptic fibrations on $X$
give the same set. Right hand sides of \eqref{exc} and \eqref{excS} give
the same. Thus, we obtain the following result which shows that
calculations of exceptional sublattices of $S_X$ for the geometric group
$\Aut X$ and for the lattice-theoretic group $A(\M(X))$
give the same.

\begin{theorem} Let $X$ be a K3 surface over an algebraically closed field,
having at least one elliptic fibration with
infinite automorphism group.

Then exceptional sublattices $E(S_X)\subset S_X$ for $\Aut X $ and for
$A(\M(X))$ are equal.
\label{thexcXS}
\end{theorem}

\medskip

We have the following general result obtained in \cite{Nik2},
\cite{Nik5}, \cite{Vin4}, \cite{Nik7} and \cite{Nik8}.

\begin{theorem} For each fixed $\rho \ge 3$, the number of hyperbolic
lattices $S$ of rank $\rho$
having non-zero exceptional sublattices $E(S)\subset S$ for $A(\M)$ is finite. 
\label{thexcfinite}
\end{theorem}

\begin{proof} For hyperbolic $E(S)$ (equivalently, when $A(\M)$ is finite),
it was proved in \cite{Nik2},
\cite{Nik5} and \cite{Vin4} (we discussed this in Sec. \ref{sec2}).
The full list of such hyperbolic lattices $S$ is known.

For parabolic $E(S)$, finiteness was proved in \cite{Nik7}, but the list of
such hyperbolic lattices $S$ is not known.

For elliptic $E(S)\not=\{0\}$, it was proved in \cite{Nik8}, but the list
of such hyperbolic lattices $S$ is not known.
\end{proof}

\medskip

Since $\rk S_X\le 22$ for K3 surfaces, using Theorem \ref{thexcfinite},
we can introduce the following finite set of even hyperbolic lattices $S$ of
$3\le \rk S\le 22$:

\begin{definition} $\mathcal SEK3$ is the set of all even hyperbolic lattices
$S$ such that $\rk S\ge 3$,
$E(S)\not=0$ for $A(\M)$ where $\M\subset \La(S)$ is a fundamental chamber
for $W^{(2)}(S)$, and $S$ is isomorphic to
the Picard lattice of some K3 surface $X$ over an algebraically closed
field.  By Theorem \ref{thexcfinite}, the set
$\mathcal SEK3$ is finite.

We denote by $\mathcal SEK3_e$, $\mathcal SEK3_p$, and $\mathcal SEK3_h$
subsets of $\mathcal SEK3$ corresponding to $A(\M)$ of
elliptic type (i.e. finite), parabolic type, and hyperbolic type
(equivalently, $E(S)=S$, $E(S)$ is semi-negative
definite and has 1-dimensional kernel, $E(S)\not=0$ is negative definite)
respectively.
\end{definition}

Combining Theorems \ref{thexcX} --- \ref{thexcfinite},
we obtain the main results.

\begin{theorem} Let $X$ be a K3 surface over an algebraically closed field,
$\rho (X)\ge 3$ and
$X$ has an elliptic fibration with infinite automorphism group.
Assume that $S_X$ is different from
lattices from the finite set $\mathcal SEK3$.

Then the exceptional sublattice $E(S_X)\subset S_X$ for $\Aut X$ is equal
to zero.

Moreover, the exceptional sublattice $E(S_X)\subset S_X$ is equal to
zero for the subgroup of $\Aut X$ generated by
automorphism groups of all elliptic fibrations on $X$
with infinite automorphism groups (or by their Mordell--Weil groups).

Moreover, we have the equality
\begin{equation}
\bigcap_{c}{(c^\perp)^{(2)}}_{pr}=E(S_X)=\{0\}.
\label{elmany}
\end{equation}
where $c$ runs through all elliptic fibrations on $X$ with infinite
automorphism groups.
\label{thAutK3}
\end{theorem}

This theorem shows that except finite number of Picard lattices from
$\mathcal SEK3$, a K3 surface $X$ has
many elliptic fibrations with infinite automorphism groups if it has one
of them: \eqref{elmany} gives the exact statement, ``how many".
We shall also discuss directly the number of elliptic
fibrations in the next section.

\medskip

It would be interesting to find the finite set of Picard lattices
$\mathcal SEK3$ of K3 surfaces. Only its
subset $\mathcal SEK3_e$ is known.

For $\rho (X)=1$, the exceptional sublattice is equal to $S_X$.
For $\rho (X)=2$, the exceptional sublattice is equal to $S_X$
if $X$ has an elliptic fibration. Indeed, in both these cases, 
$\Aut X$ is finite since $O(S_X)$ is finite (this was observed
in \cite{PS}). Thus, only the case of $\rho (X)\ge 3$ which we
considered above is interesting.


\section{Number of elliptic fibrations and\\
elliptic fibrations with  infinite\\
automorphism groups on K3 surfaces}\label{sec4}

Using Theorem \ref{thAutK3}, we obtain the following results which
show that for $\rho(X)\ge 3$, K3
surface $X$ has infinite number of elliptic fibrations and infinite number
of elliptic fibrations
with infinite automorphism groups if it has one of them, if $S_X$ is
different from a finite number of exceptional Picard
lattices.

\begin{theorem} Let $X$ be a K3 surface over an algebraically closed field,
$\rho (X)\ge 3$ and $X$ has
at least one elliptic fibration.

Then $X$ has infinite number of elliptic fibrations if $S_X$ is different
from the following finite set of Picard latices $S_X$ when the number
of elliptic fibrations is finite:

$S_X\in {\mathcal SEK3}_e$.  In particular, $\Aut X$ is finite.

$S_X\in {\mathcal SEK3}_p$ and $X$ has only one elliptic fibration.
In particular, $X$ has one elliptic fibration with infinite automorphism group,
and no other elliptic fibrations.
\label{thnumberelK3}
\end{theorem}

\begin{proof} Let us assume that $X$ has finite number of elliptic fibrations.
Then $\M(X)$ has only finite number of fundamental primitive
isotropic elements which are all exceptional for $A(\M(X))$.
It follows that $E(S_X)\not=0$ and $S_X\in {\mathcal SEK3}$.
Let us consider two cases.

{\it Case 1:} Let us assume that all elliptic fibrations on $X$ have
finite automorphism groups (equivalently, fundamental primitive
isotropic elements
$c$ for $\M(X)$ have finite stabilizer subgroups $A(\M(X))_c$).
Then $A(\M(X))$ is finite and $S_X\in {\mathcal SEK3}_e$.
Vise a versa, if $S_X\in {\mathcal SEK3}_e$ then $\M(X)$ is a
fundamental chamber for the arithmetic group $W^{(2)}(S_X)$ in $\La(S_X)$,
and it has only finite number of fundamental primitive isotropic elements.
Thus, $X$ has only finite number of elliptic fibrations.

{\it Case 2:} Let us assume that $X$ has an elliptic fibration
$c$ with infinite automorphism group $\Aut (c)$. Since the number
of elliptic fibrations on $X$ is
finite, all of them are exceptional for $\Aut X$, and $E(S_X)$ is not zero.
Since $\Aut X$ is infinite,
$E(S_X)$ cannot be hyperbolic. Since elliptic fibrations give isotropic
elements, then $E(S_X)$ cannot
be elliptic (i.e., negative definite). Thus, $E(S_X)$ is parabolic
and $S_X\in {\mathcal SEK3}_p$. By Theorem
\ref{thexcX}, we have
$$
E(S_X)=
\bigcap_{c}{(c^\perp)^{(2)}}_{pr}
$$
where $c$ runs through all elliptic fibrations on $X$ with infinite
automorphism groups. Since $E(S_X)$ is parabolic, it follows that $X$
has only one elliptic fibration $c$
with infinite automorphism group and $\Aut X=\Aut(c)$. Since all
elliptic fibrations on $X$ are exceptional
for $\Aut X$, they also must have infinite automorphism groups,
and they must be equal
to $c$. Thus, $X$ has only one elliptic fibration $c$.

Vice versa, let us assume that $S_X\in {\mathcal SEK3}_p$ and $X$ has
only one elliptic fibration. Then $E(S_X)\not=S_X$ and $\Aut X$
is infinite. Since $X$ has only one elliptic fibration $c$,
then $\Aut X=\Aut(c)$ is its automorphism group which is
infinite.

This finishes the proof.
\end{proof}

\begin{theorem} Let $X$ be a K3 surface over an algebraically closed field,
$\rho (X)\ge 3$ and $X$ has
at least one elliptic fibration with infinite automorphism group.

Then $X$ has infinite number of elliptic fibrations with infinite
automorphism groups if $S_X$ is
different from the following finite set of Picard latices $S_X$
when the number
of elliptic fibrations on $X$ with infinite automorphism groups is finite:

$S_X\in {\mathcal SEK3}_p$.  In particular, $X$ has only one
elliptic fibration with infinite automorphism group.
\label{thnumberelinfK3}
\end{theorem}

\begin{proof} Since the number of elliptic fibrations on $X$
with infinite automorphism groups is finite, all of them
are exceptional for $\Aut X$, and $E(S_X)$ is not trivial.
Then $S_X\in {\mathcal SEK3}$ which is finite
by Theorem \ref{thAutK3}.
Since each of these elliptic fibrations is exceptional for $\Aut X$
and corresponds to an isotropic element, then $E(S_X)$ cannot be
elliptic (that is negative definite). Since $\Aut X$ is infinite,
$E(S_X)$ cannot be hyperbolic either. Thus,
$E(S_X)$ is parabolic and has 1-dimensional kernel. By Theorem  \ref{thexcX},
$$
E(S_X)=
\bigcap_{c}{(c^\perp)^{(2)}}_{pr}
$$
where $c$ runs through all elliptic fibrations on $X$ with infinite
automorphism groups. Since $E(S_X)$ is parabolic, it follows that $X$
has only one elliptic fibration $c$
with infinite automorphism group.

Vice a versa, if $X$ has only one elliptic fibration $c$ with infinite
automorphism group, then
$E(S_X)=(c^\perp)^{(2)}_{pr}$ is parabolic and $S_X\in {\mathcal SEK3}_p$.

It follows the statement.
\end{proof}

If $\rho (X)=1$ or $2$, then $X$ has less or equal to two elliptic
fibrations, and these cases are trivial.


\section{Applications to K3 surfaces with\\  
exotic structures}
\label{sec5}

Here we want to give some other applications of finiteness of the set
of Picard lattices of K3 surfaces with non-trivial exceptional
sublattice $E(S_X)$, and elliptic fibrations with
infinite automorphism group.

\subsection{K3 surfaces with finite number of non-singular rational
curves} Recently, D. Matsushita asked me what we can say
about K3 surfaces with finite number of non-singular rational
(equivalently, irreducible $(-2)$-curves). We have the
following

\begin{theorem} A K3 surface $X$ over an algebraically closed field
has no non-singular rational curves if and only if its Picard lattice $S_X$
has no elements with square $(-2)$.

If a K3 surface $X$ over an algebraically closed field has
non-singular rational curves
(equivalently, its Picard lattice $S_X$ has
elements with square $(-2)$), then their number is finite in the
following and only the following cases (1) and (2):

(1) $\rho (X)=2$;

(2) $\rho (X)\ge 3$ and the Picard lattice $S_X$ is
elliptically $2$-reflective: $[O(S_X):W^{(2)}(S_X)]<\infty$.
The number of elliptically $2$-reflective
hyperbolic lattices is finite, and they are
enumerated in \cite{Nik2}, \cite{Nik5} and \cite{Vin4}
(for $\rho (X)\ge 5$ see their lists in Sec. \ref{sec2}).
\label{th:ratcurves}
\end{theorem}

\begin{proof} Let $S=S_X$ and $\M=\M(X)$ be the fundamental chamber for
$W^{(2)}(S)$. Then the non-singular rational curves on $X$ are in one
to one correspondence to elements of the set $P(\M)$ of
perpendicular vectors to $\M$ with square $(-2)$ and directed outwards.

If $S$ has no elements with square $(-2)$, then $P(\M)$ is empty and
$X$ has no non-singular rational curves.

Let us assume that $S$ has an element $\delta$ with $\delta^2=-2$.
Then $\pm w (\delta)$ gives one of the elements of $P(\M)$ for some
$w\in W^{(2)}(S)$, the set $P(\M)$ is not empty,
and $X$ contains a non-singular rational curve.

Since $S$ is hyperbolic and $S$ has elements with square $(-2)$, then
$\rho (X)=\rk S\ge 2$.

Let $\rho (X)=2$. Then $\M=V^+(S)/\br^+$ is an interval, and elements of
$P(\M)$ correspond to terminals of this interval.
Thus, $P(\M)$ has not more than $2$ elements, and the number
of non-singular rational curves on $X$ is one or two.

Let $\rho (X)\ge 3$, and $P(\M)$ is non-empty and finite. Then all elements
of $P(\M)$ are exceptional for the symmetry group $A(\M)\subset O^+(S)$ of
$\M$, and the exceptional sublattice $E(S)$ is not zero. Then by
Theorem \ref{thexcfinite} (from \cite{Nik2},
\cite{Nik5}, \cite{Vin4}, \cite{Nik7} and \cite{Nik8}), 
$S$ is one of a finite number of hyperbolic lattices of rank $\le 22$.

Actually, the main idea of the proof of this theorem in
\cite{Nik5}, \cite{Nik7} and \cite{Nik8} is that $P(\M)$ has
elements $\delta_1,\dots \delta_\rho\in S$, $\rho=\rk S$, which
generate $S\otimes \bq$ and $\delta_i\cdot \delta_j\le 19$,
$1\le i<j\le \rho$. (These elements define a narrow part of $\M$.)
It follows that $P(\M)$ generates $S\otimes \bq$, the group
$A(\M)\cong O^+(S)/W^{(2)}(S)$ is finite since $P(\M)$ is finite,
the lattice $S$ is elliptically $2$-reflective: $[O(S):W^{(2)}(S)]< \infty$, 
and the number of such lattices is finite.

This finishes the proof.
\end{proof}

\subsection{K3 surfaces with finite number of Enriques involutions}

Here we restrict to basic fields $k$ of $\cha k\not=2$.

We recall that an involution $\sigma$ on a K3 surface $X$ over an algebraically
closed field $k$ of $\cha k\not=2$
is called {\it Enriques involution} if $\sigma$ has no fixed points on
$X$. Then $X/\{id,\sigma\}$ is Enriques surface. See \cite{CD}.
It is well-known (see \cite{CD}) that $\sigma$ in $S_X$
has the eigen-value $1$ part which is isomorphic to the standard
hyperbolic lattice $S_X^\sigma\cong U(2)\oplus E_8(2)$ of rank $10$.
A general K3 surface with Enriques involution has
$S_X=S_X^\sigma\cong U(2)\oplus E_8(2)$, and only finite number of Enriques
involutions (if $\cha k=0$, it is unique).

We have the following result.

\begin{theorem} Let $X$ be a K3 surface over an
algebraically closed field $k$ of $\cha k\not=2$ and $X$ has an
Enriques involution.

If $X$ has only finite number of Enriques involutions,  
then either $S_X$
is isomorphic to $U(2)\oplus E_8(2)$, or $S_X$ belongs
to the finite set ${\mathcal SEK3}$.  

In particular, if $S_X$ is different from lattices of these two
finite sets, then $X$ has infinite number of Enriques involutions.  
\label{th:Enrinv}
\end{theorem}

\begin{proof} Let $\sigma$ be an Enriques involution on $X$.

Since $S_X^\sigma\cong U(2)\oplus E_8(2)$ is a sublattice of $S_X$,
it follows that $\rho (X)\ge 10$.

Let $\rho (X)=10$. Then $S_X=S_X^\sigma$, and $\sigma$ is the identity
on $S_X$. Since $\Aut X$ has only finite kernel in $S_X$, it follows
that $X$ has only finite number of Enriques involutions
(it is unique if $\cha k=0$).

Let $\rho (X)>10$. Then for each Enriques involution $\sigma$ on $X$, the
orthogonal complement $S_\sigma=(S_X^\sigma)^\perp$ in $S_X$
is a non-zero negative definite sublattice of $S_X$ which has a finite
automorphism groups. If $X$ has only finite
number of Enriques involutions, then all these orthogonal complements
are contained in the exceptional sublattice $E(S_X)$ of $S_X$ for 
$\Aut X$, and $E(S_X)\not=\{0\}$. 
Since $\rho (X)\ge 10\ge 6$, then by Theorem \ref{elK36} either 
$\Aut X$ is finite and $X$ has only finite number 
of Enriques involutions, and $S_X\in {\mathcal SEK3}_e$, 
or $X$ has an elliptic fibration with infinite automorphism group. 
By Theorem \ref{thexcXS},  
then exceptional sublattices of $S_X$ for $\Aut X$ and for $A(\M(X))$ 
are the same,  
and $S_X\in {\mathcal SEK3}$. By Theorem \ref{thexcfinite} (from \cite{Nik2},
\cite{Nik5}, \cite{Vin4}, \cite{Nik7} and \cite{Nik8}), the set 
${\mathcal SEK3}$ is finite.

This proves the theorem.
\end{proof}

The method of the proof is so general, that by the same considerations,
one can prove similar results for other types of involutions or automorphisms
on K3 surfaces, and other structures on K3 surfaces.

\subsection{K3 surfaces with naturally arithmetic\\ automorphism groups}

This is related to the recent preprint by B. Totaro \cite{T}.

\begin{definition} Let $X$ be a K3 surface over an
algebraically closed field, and $S_X$ its Picard lattice.

We say that the automorphism group $\Aut X$ is {\it naturally arithmetic,}
if there exists a sublattice $K\subset S_X$ such that the action of
$\Aut X$ in $S_X$ identifies $\Aut X$ as a subgroup
of finite index in $O(K)$. More precisely, there exists a subgroup
$G\subset \Aut X$ of finite index such that $K$ is $G$-invariant, and the
natural homomorphism $G\to O(K)$ has finite kernel and cokernel.
\label{natarith}
\end{definition}

For example, if $\Aut X$ is finite, then one can take $K=\{0\}\subset S_X$,
and $\Aut X$ is naturally arithmetic. Thus, all K3 surfaces with
elliptically 2-reflective Picard lattices $S_X$
(for $\rho (X)\ge 5$ see there list in Sec. \ref{sec2})
have naturally arithmetic automorphism groups.

We have the following result which uses the Global Torelli Theorem
for K3 surfaces \cite{PS}, and it is valid over $\bc$ (or over an 
algebraically closed field $k$ of $\cha k=0$).

\begin{theorem} Let $X$ be a K3 surface over $\bc$. Then $\Aut X$
is naturally arithmetic in the following and only the following
cases (1), (2) and (3):

(1) The Picard lattice $S_X$ has no elements with square $(-2)$.

(2) $S_X$ has elements with square $(-2)$ and $\rho (X)=2$.

(3) $S_X$ has elements with square $(-2)$, $\rho (X)\ge 3$, and
$S_X$ is one of lattices from the subset (it will be described in the proof) 
of the finite set ${\mathcal SEK3}$. 

In particular, if $\rho (X)\ge 3$ and $S_X$ has elements with square $(-2)$,
then $\Aut X$ is not naturally arithmetic, except a finite number of Picard
lattices $S_X$. 
\label{th:natarith}
\end{theorem}

\begin{proof} We identify $\Aut X$ with its action in $S_X$.
By \cite{PS}, the automorphism group $\Aut X$ is a subgroup of
finite index in $A(\M)$ where $\M=\M(X)$, and
$O^+(S_X)=A(\M)\ltimes W^{(2)}(S_X)$ is the semi-direct product.
We can consider $A(\M)$ instead of $\Aut X$. Thus, the natural 
arithmeticity of $\Aut X$ depends on $S_X$ only. 
If $S_X$ has no elements with square $(-2)$, then $W^{(2)}(S_X)$
is trivial, and $A(\M)$ and $\Aut X$ are naturally arithmetic
(one can take $K=S_X$). We obtain the case (1).

If $S_X$ has elements with square $(-2)$ and $\rho (X)=2$,
then $A(\M)$ and $\Aut X$ are finite, and they are naturally arithmetic
(one can take $K=\{0\}$). We obtain the case (2).

Let us assume that $S_X$ has elements with square
$(-2)$, $\rho (X)\ge 3$ and $\Aut X$ is naturally arithmetic for
some sublattice $K\subset S_X$. Let us show that then 
the exceptional sublattice $E(S_X)$ for $A(\M)$ (or $\Aut X$)
is not trivial.

Let us assume that $W^{(2)}(S_X)$ is finite. The group $W^{(2)}(S_X)$
is generated by reflections $s_\delta$ where $\delta \in P(\M)$,
and all such reflections are different. It follows that
$P(\M)$ is finite and non-empty (equivalently, the number of non-singular
rational curves on $X$ is finite and non-empty).
By Theorem \ref{th:ratcurves}, then $S_X$ is elliptically $2$-reflective:
$[O(S_X):W^{(2)}(S_X)]<\infty$, the groups $A(\M)$ and $\Aut X$
are finite, and the exceptional sublattice $E(S_X)=S_X$ is not trivial.

Let us assume that $W^{(2)}(S_X)$ is infinite. Then $W^{(2)}(S_X)$
and $O^+(S_X)$ act transitively on infinite number of
fundamental chambers for $W^{(2)}(S_X)$ in $\La(S_X)$.
But $A(\M)$ sends $\M$ to itself. Thus, $A(\M)$ has infinite index in
$O^+(S_X)$. It follows that $K\subset S_X$ has $\rk K<\rk S_X$, and
the orthogonal complement $E=K^\perp$ in $S_X$ is not zero.

If $K$ is negative definite, then $A(\M)$ and $\Aut X$ are finite, and
$E(S_X)=S_X$ is not trivial.

If $K$ is semi-negative definite and not negative definite, 
then it has a one-dimensional kernel   
$\bz c$, where $c$ is exceptional, and $E(S_X)$ is not trivial. 

If $K$ is hyperbolic, then $E=K^\perp$ is negative definite and non-zero. 
It gives a non-trivial sublattice in $E(S_X)$ since $E$ has a finite
automorphism group. Thus, $E(S_X)$ is not trivial.

By Theorem \ref{thexcfinite}, the lattice $S_X$ is one of
a finite number of hyperbolic lattices $S$ with $3\le \rk S \le  22$ and
with non-trivial exceptional sublattice $E(S)$ for $A(\M)$. Thus, $S_X$ 
belongs to the finite set ${\mathcal SEK3}$ of hyperbolic lattices. 

This proves the theorem.
\end{proof}

Because of Theorem \ref{th:natarith}, the following result is important. 
We know it for many years, and it is a corollary of results of \cite{Nik1}. 
As we know, it was never published. 

\medskip 

\begin{theorem} Let $X$ be a K3 surface over $\bc$, and 
$\rho(X)=\rk S_X\ge 12$. 

Then $S_X$ has elements with square $(-2)$. 
In particular, $X$ contains a non-singular rational curve $\bp^1$.
\label{th:K3rhoge12}
\end{theorem}

\begin{proof} The Picard lattice $S=S_X$ is a primitive sublattice 
of the lattice $H^2(X,\bz)=L$ which is an even unimodular 
lattice of signature $(3,19)$. It is unique up to isomorphisms. 
The transcendental lattice 
$T=S^\perp_L$ has rank $\le 10$, and $\rk T+\rk (T^\ast/T)\le 20=\rk L -2$.   
By \cite[Theorem 1.14.4]{Nik1}, the lattice $T$ has a unique primitive 
embedding into $L$,  up to isomorphisms. 
Thus, for any primitive embedding $T\subset L$, 
we have $(T)^\perp_L\cong S=S_X$. 

On the other hand, by \cite[Theorem 1.12.2]{Nik1}, the 
lattice $T\oplus \langle -2 \rangle$ 
has a primitive embedding into 
$L$. For this primitive embedding, $T^\perp_L$ contains a sublattice 
$\langle -2\rangle$. Thus, $S=S_X$ also contains a primitive sublattice 
$\langle -2 \rangle$. Equivalently, there exists $\delta \in S_X$ with 
$\delta^2=-2$. 

This proves the theorem.  
\end{proof}

From Theorems \ref{th:natarith} and \ref{th:K3rhoge12}, we 
obtain  

\begin{corollary} Up to isomorphisms, there exists 
only a finite number of Picard lattices $S_X$ of K3 surfaces over $\bc$ 
such that $\rk S_X=\rho(X)\ge 12$ and $\Aut X$ is naturally arithmetic. 
\label{cor:natarrhoge12}
\end{corollary}

In contrary, by \cite[Theorem 1.12.4]{Nik1}, any even hyperbolic 
lattice $S$ of $\rk S\le 11$ 
has a primitive embedding into even unimodular lattice of signature $(3,19)$. 
Thus, by epimorphicity of Torelli map for K3 surfaces, \cite{Kul}, the 
lattice $S$  is isomorphic to Picard lattice $S_X$ of a 
K3 surface $X$ over $\bc$. It follows that for each $1\le \rho\le 11$, 
there exists infinite number of non-isomorphic  
Picard lattices $S_X$ of K3 surfaces over $\bc$ such that 
$\rk S_X=\rho$, $S_X$ has no elements with square $(-2)$ and then 
$\Aut X$ is naturally arithmetic.


V.V. Nikulin \par Deptm. of Pure Mathem. The University of
Liverpool, Liverpool\par L69 3BX, UK; \vskip1pt Steklov
Mathematical Institute,\par ul. Gubkina 8, Moscow 117966, GSP-1,
Russia

vnikulin@liv.ac.uk \ \ vvnikulin@list.ru


\begin{thebibliography}{ADSE}

\bibitem{Can} S. Cantat, 
\emph{Dynamique des automorphismes des surfaces $K3$}, 
Acta Math.  \textbf{187}  (2001),  no. 1, 1--57. 


\bibitem{CD} F.R. Cossec, I.V. Dolgachev, \emph{Enriques surfaces I},
Progress in Mathematics, V. 76,
Birkh\"auser, 1989, 387 pages.

\bibitem{Kul} Vic. S. Kulikov,
\emph{Degenerations of $K3$ surfaces and Enriques surfaces},
Izv. Akad. Nauk SSSR Ser. Mat. \textbf{41} (1977), no. 5, 1008--1042;
English transl. in  Math. USSR Izv. \textbf{11}, (1977) no. 5, 957--989.


\bibitem{Nik1} V.V. Nikulin,
\emph{Integral symmetric bilinear forms and some of their geometric
applications},
Izv. Akad. Nauk SSSR Ser. Mat. \textbf{43} (1979), 111--177;
English transl. in  Math. USSR Izv. \textbf{14} (1980).

\bibitem{Nik2} V.V. Nikulin,
\emph{On the quotient groups of the automorphism groups of hyperbolic forms
by the subgroups generated by 2-reflections, Algebraic-geometric applications},
Current Problems in Math. Vsesoyuz.
Inst. Nauchn. i Techn. Informatsii, Moscow \textbf{18}
(1981), 3--114;
English transl. in J. Soviet Math. \textbf{22} (1983), 1401--1476.

\bibitem{Nik3} V.V. Nikulin,
\emph{On arithmetic groups generated by reflections in Lobachevsky spaces},
Izv. Akad. Nauk SSSR Ser. Mat. \textbf{44} (1980), 637--669;
English transl. in  Math. USSR Izv. \textbf{16} (1981).

\bibitem{Nik4} V.V. Nikulin, \emph{On the classification of arithmetic groups
generated by reflections in Lobachevsky spaces}, Izv. Akad. Nauk SSSR Ser.
Mat. \textbf{45} (1981), 113--142;  English transl. in
Math. USSR Izv. \textbf{18} (1982).

\bibitem{Nik5} V.V. Nikulin,
\emph{Surfaces of type K3 with finite automorphism group
and Picard group of rank three},
Proc. Steklov Math. Inst. \textbf{165} (1984), 113--142;
English transl. in  Trudy Inst. Steklov \textbf{3} (1985).

\bibitem{Nik6} V.V. Nikulin,
\emph{Discrete reflection groups in Lobachevsky spaces
and algebraic surfaces}, Proc.  Int. Congr.
Math. Berkeley 1986, Vol. 1, pp. 654-669.

\bibitem{Nik7} V.V. Nikulin,
\emph{Reflection groups in Lobachevsky spaces and the denominator
identity for Lorentzian Kac--Moody algebras},
Izv. Akad. Nauk of Russia. Ser.
Mat. \textbf{60} (1996), 73--106;  English transl. in
Russian Acad. Sci. Izv. Math. \textbf{60} (1996)
(see also alg-geom/9503003).


\bibitem{Nik8} V.V. Nikulin,
\emph{K3 surfaces with interesting groups of automorphisms},
Algebraic Geometry 8, J. Math. Sci. (New York) \textbf{95} (1999),
no. 1, 2028--2048
(see also alg-geom/9701011).

\bibitem{PS} I.I. Pjatetski\u i-\u Sapiro and I.R. \u Safarevi\u c,
\emph{A Torelli theorem for algebraic surfaces of type K3}, Izv. AN SSSR.
Ser. mat., \textbf{35} (1971), no. 3, 530--572;
English transl.: Math. USSR Izv. \textbf{5} (1971), no. 3, 547--588.

\bibitem{RS} A.N. Rudakov and I.R. Shafarevich,
\emph{Surfaces of type K3 over fields of finite characteristic},
Current Problems in Math. Vsesoyuz. Inst. Nauchn. i Techn. Informatsii,
Moscow \textbf{18}
(1981), 115--207;
English transl. in J. Soviet Math. \textbf{22} (1983), 1476--1533.

\bibitem{S} I.R. Shafarevich (ed.), \emph{Algebraic Surfaces},
Proc. Steklov Math. Inst.
\textbf{75}, (1965), 3--215.

\bibitem{T} B.Totaro, \emph{Algebraic surfaces and hyperbolic geometry},
Preprint 2010, 19 pages, arXiv:1008.3825v1.

\bibitem{Vin1} E.B. Vinberg, \emph{Discrete groups generated by
reflections in Loba\v cevski\v i spaces}, Mat. Sb. (N.S.) \textbf{72}
(1967), 471--488; English transl. in Math. USSR Sb.\textbf{1} (1967),
429--444.

\bibitem{Vin2} E.B. Vinberg, \emph{Absence of crystallographic
reflection groups in Lobachevski\v i spaces of large dimension},
Trudy Moskov. Mat. Obshch. \textbf{47} (1984), 68--102; English transl.
in Trans. Moscow Math. Soc. \textbf{47} (1985).

\bibitem{Vin3} E.B. Vinberg
\emph{Hyperbolic groups of reflections},
Uspekhi Mat. Nauk \textbf{40} (1985), no. 1, 29--66;
English transl. in Russian. Math. Surv. \textbf{40} (1985), no. 1,
31--75.


\bibitem{Vin4} E.B. Vinberg,
\emph{Classification of 2-reflective hyperbolic lattices of rank 4},
Tr. Mosk. Mat. Obs. \textbf{68} (2007), 44--76;
English transl. in Trans. Moscow Math. Soc. (2007), 39--66.





\end{thebibliography}
\end{document}